\documentclass{article}
\usepackage{fullpage}
\usepackage{url}
\usepackage[utf8]{inputenc}
\usepackage{xcolor}
\usepackage{xfrac}
\usepackage{pgfplots}
\usepackage{wrapfig}
\usepackage[export]{adjustbox}
\pgfplotsset{compat=newest}

\usepgfplotslibrary{external}
\usepackage{parskip} 
\usepackage{amsthm}
\usepackage{amsmath}
\usepackage{amssymb}
\usepackage{bbm}
\usepackage{physics}
\usepackage{mathtools}
\usepackage{dsfont}
\usepackage{enumitem}
\usepackage{tikz}
\usepackage{wrapfig, blindtext}

\DeclarePairedDelimiter{\floor}{\lfloor}{\rfloor}

\newcommand{\Z}{\mathbb{Z}}
\newcommand{\Q}{\mathbb{Q}}
\newcommand{\R}{\mathbb{R}}
\newcommand{\C}{\mathbb{C}}

\newcommand{\HH}{\mathbb{H}}

\newcommand{\ind}{\mathbbm{1}}
\newcommand{\eps}{\varepsilon}

\newtheorem{thm}{Theorem}

\newtheorem{prop}{Proposition}
\newtheorem{lemma}{Lemma}

\newtheorem*{claim*}{Claim}

\newtheorem*{remark}{Remark}
\newtheorem{corollary}{Corollary}
\newtheorem{conjecture}{Conjecture}

\title{Moments of the Zeros of Faber Polynomials of the Miller Basis}
\date{}
\author{Adi Zilka\footnote{E-mail: adizilka@mail.tau.ac.il}}
\begin{document}
\maketitle

\begin{abstract}
    We study the zeros of modular forms in the Miller basis, a natural basis for the space of modular forms.
    We show that the zeros of their Faber polynomials have linear moments. By analysing the moments, we can extend the known range of the forms in the Miller basis for which at least one of the zeros is not on the arc - the circular part of the boundary of the fundamental domain. Additionally, for forms in the Miller basis of an index asymptotically linear in the weight such that all zeros are on the arc, we compute the limit distribution of the zeros, which depends on the asymptotic ratio of the index to the weight. 
\end{abstract}

\section{Introduction and Main Results}

For $k \ge 0$ an even integer, let $M_k$ be the space of modular forms of weight $k$ for the full modular group $SL_2(\Z)$. Each $f \in M_k$ has the $q$-expansion 
$$f(\tau) = \sum_{n=0}^{\infty} a_f(n) q^n$$
where $q = e^{2 \pi i \tau}$, $\tau \in \mathbb{H} = \{ \tau : \Im{\tau} > 0 \}$. 
\\\\
Writing $k = 12\ell + k'$, $k' \in \{0,4,6,8,10,14\}$, the \textit{Miller basis} of $M_k$ consists of the unique elements $f_{k,m}$ of $M_k$ with the $q$-expansion
$$f_{k,m} = q^m + O(q^{\ell+1})$$
for $m = 0,...,\ell$.
\\\\
Our goal is to understand the zeros of the forms in the Miller basis. An important method for studying the zeros was first used for the Eisenstein series in 1970 by F. Rankin and Swinnerton-Dyer \cite{Eisenstein zeros}. They showed that the zeros in the fundamental domain of the Eisenstein series $E_{k}$ all lie on the arc $\mathcal{A}=\{e^{i\theta}:\theta\in [\frac{\pi}{2},\frac{2\pi}{3}]\}$, and become uniformly distributed in $\mathcal{A}$ as $k\to\infty$. W. Duke and P. Jenkins \cite{first miller basis} showed that the zeros of $f_{k,0}=1+O(q^{\ell+1})$ all lie on the arc $\mathcal{A}$ and become uniformly distributed. Their result was later extended by Raveh \cite{Roei's paper} who showed that all the zeros of $f_{k,m}$ are on the arc for $m < \frac{2}{9}\ell$ and showed the zeros are uniformly distributed for $m = o(\ell)$. 
\\\\
The main tool we will use to study the zeros are Faber polynomials. We write $F_{k,m}(t) = e_0 t^{\ell - m} + e_1 t^{\ell - m - 1} + ... + e_{\ell - m}$ for the Faber polynomial of $f_{k,m}$. The Faber polynomial is the polynomial that satisfies
$$f_{k,m} = \Delta^{\ell}\cdot F_{k,m}(j)\cdot E_{k'}$$
where $\Delta$ is the modular discriminant and $j$ is Klein’s absolute invariant. \\\\
Take $x_i = x_i(k,m)$ to be the zeros of the Faber polynomial, counted with multiplicity. Then $j^{-1}(x_i)$ are the zeros of the modular form and $j$ is a bijection from the fundamental domain to $\C$. This means that we can study the zeros of a modular form by studying the zeros of its Faber polynomial. Note that $j(\mathcal{A}) = [0,1728]$, $j$ maps the imaginary axis to $[1728, \infty)$ and the left boundary segment to the negative real numbers.
\\\\
Using Faber polynomials, Rudnick \cite{large m} showed that if the degree of the Faber polynomial $\ell - m = D$ is fixed as $k \to \infty$ the zeros of $f_{k,m}$ lie on D vertical lines in the fundamental domain and are of approximate height $\log(k)$.
\\\\
Our goal is to study the case where the degree of the polynomial is not bounded. It turns out the zeros do have a structure that can be analyzed with no assumptions on the degree.
\\\\
Recall that we take $x_i = x_i(k,m)$ to be the zeros of the Faber polynomial, counted with multiplicity.
\begin{thm}\label{linearity thm}
    For every $1 \le n \le \ell - m$, there exists $A_n, B_n, C_n(k')$ such that 
    $$\sum_{i=1}^{\ell - m} x_i^n = A_n\cdot k + B_n \cdot m + C_n(k')$$
    where $A_n = \frac{1}{2\pi} \int_{\mathcal{A}} j^n(e^{i\theta})d\theta$, $-B_n$ is the coefficient of $q^0$ in $j^n$, $C_n(0) = C_n(4) = C_n(8) = 0$ and $C_n(6) = C_n(10) = C_n(14) = -\frac{1728^n}{2}$. 
\end{thm}
The proof of Theorem \ref{linearity thm} uses only basic facts about Faber polynomials, like the fact that the leading coefficients of $F_{24 \ell, 2m}$ are identical to the leading coefficients of $F_{12\ell, m}^2$. The value of $A_n$ is a direct corollary of a result from W. Duke and P. Jenkins \cite{first miller basis}, which they proved by analytic methods. The value of $B_n$ can be derived from a recent paper by R. Raveh \cite{Roei's paper} proven by similar methods. We also give a purely algebraic proof, which has the added benefit of working in a more general setting. 
\\\\
This theorem gives a very structured constraint on the roots of the Faber polynomials and can be used in proving statements about the roots. The sums of the powers of the roots hold all the information about the coefficients of the polynomial (using Newton's identities we can go back and forth between them), but in some cases are much simpler to work with. From noticing that if all the zeros of the form are on the arc these sums are positive, we get the following corollary.

\begin{corollary}
    If $k' = 0,4,8$ and $\frac{30}{31}\ell + k'\frac{5}{62} < m \le \ell - 1$, at least one of the zeros of $f_{k,m}$ is not on the arc.
    \\\\
    If $k' = 6,10,14$ and $\frac{30}{31} \ell + k' \frac{5}{62}  - \frac{36}{31} < m \le \ell - 1$, at least one of the zeros of $f_{k,m}$ is not on the arc.
\end{corollary}
Using Newton's identities, we obtain a new formula for the coefficients of the Faber polynomials of modular forms in the Miller basis. There are no previously known non-recursive formulas for the coefficients.
\begin{corollary}
    Let $F_{k,m}(t) = e_0 t^{\ell - m} + e_1 t^{\ell - m - 1} + ... + e_{\ell - m}$ be the Faber polynomial of the modular form $f_{k,m}$. The coefficient $e_n$ is given by
    $$e_n = \sum_{\substack{t_1 + 2 t_2 + ... + n t_n = n \\ 0 \le t_s}} \prod_{s = 1}^n (-1)^{t_s} \frac{(A_s \cdot k + B_s \cdot m + C_s(k'))^{t_s}}{t_s ! s^{t_s}}.$$
\end{corollary}
The power sums also hold information on the distribution of the roots. As the $n$-th moment of the distribution of the roots of $F_{k,m}$ is 
$$M_n(k,m) = \frac{1}{\ell - m} \sum_{i=1}^{\ell - m} x_i^n = \frac{A_n \cdot k + B_n \cdot m + C_n(k')}{\ell- m } = \frac{A_n (12 + \frac{k'}{\ell}) + B_n \cdot \frac{m}{\ell} + \frac{C_n(k')}{\ell}}{1 - \frac{m}{\ell} },$$
from bounds on the moments, as long as all the roots are on $\R$ (all the zeros of the modular form are on the boundary or on $i \R$), the moments uniquely define the distribution. Since in the limit the moments depend only on $\frac{m}{\ell}$, the limit distribution depends only on $\lim_{\ell \to \infty} \frac{m}{\ell}$.
\\\\
The moments can also be used to determine the limit distribution of the zeros of the modular form on the arc.
\begin{thm}\label{limit dist}
    For  $\frac{m}{\ell} \to c$, $0 < c < 1$ for which all the zeros of $f_{k,m}$ are on the arc, as $\ell \to \infty$ the limit probability that a zero of the modular form on the arc is between $[\theta_1, \theta_2]$ is
    $$\frac{6}{\pi(1-c)}(\theta_2-\theta_1)+\frac{2c}{1-c} (\cos(\theta_2)-\cos(\theta_1)),$$
    for any $\frac{\pi}{2} \le \theta_1 \le \theta_2 \le \frac{2\pi}{3}$.
\end{thm}
Using a result of Raveh \cite{Roei's paper} we can take $c < \frac{2}{9}$.
\\\\
Thus in our range, the limit distribution of the zeros depends on the asymptotic ratio of the index to the weight on the interval $ [\frac{\pi}{2},\frac{2\pi}{3}]$, unlike the uniform distribution in the work of Duke and Jenkins and of Raveh.
\\\\
Numerical investigation of these polynomials revealed further remarkable properties. For example,
\begin{conjecture}\label{5/8}
    For $\ell >> 1$, let $m(k)$ be the minimal $m$ for which $f_{k,m}$ has at least one zero not on the arc. Then $m(12 \ell) = 10 s + m_r$ and if $\ell > m \geq m(12\ell)$ at least one of the zeros is not on the arc, where we denote $m_0, m_1, ..., m_{15} = 4, 3, 4, 5, 6, 7, 6, 7, 8, 9, 10, 9, 10, 11, 12, 13$, and $\ell = 16 s + r$ with $0 \le r < 16$. 
\end{conjecture}
\begin{conjecture}\label{no big zeros}
    The Faber polynomials of modular forms in the Miller basis have no zeros in $[1728, \infty)$.
\end{conjecture}
\section*{Organization of the paper}
Part 2 will introduce the needed background, including how to compute Faber polynomials.
\\\\
In Part 3 we will prove Theorem \ref{linearity thm}, computing $A_n$, $B_n$ and $C_n$ and give a new formula for the coefficients of the Faber polynomials.
\\
The computations of $B_n$ will be algebraic and only based on the coefficients of $\Delta$ and the $j$-function.
\\\\
Part 4 will prove the upper bound that can be achieved using Theorem \ref{linearity thm} on the value of $m$ for which all the zeros of $f_{k,m}$ are on the arc.
\\\\
Part 5 will discuss the limit distribution for the zeros of $f_{k,m}$ in the case of  $\frac{m}{\ell} \to c >0$ using Theorem \ref{linearity thm} and an additional proof using Raveh \cite{Roei's paper}. This result can be used for an additional proof of the formula for $B_n$.
\\\\
\textbf{Acknowledgements:}
\\\\
This research was supported by the Israel Science Foundation (grant No. 2860/24).
\\\\
This work is part of my M.Sc. thesis. I would like to thank my advisor, Zeev Rudnick, for all of his advice, helpful suggestions and immense support during the process of research and writing of this paper.
\\\\
I also thank Alon Nishry for many helpful comments.

\section{Background}

Denote by $\HH=\{\tau:\Im(\tau)>0\}$ the upper half-plane.
\\\\
Let $k\ge 4$ be an even integer. A modular form of weight $k$ is a holomorphic function $f:\HH\to\C$, such that 
$$
    f\left(\frac{a\tau+b}{c\tau +d}\right)=\left(c\tau+d\right)^{k}f(\tau)
$$
for all 
$\begin{psmallmatrix}
    a & b \\
    c & d
    \end{psmallmatrix} \in \text{SL}_{2}(\Z)
$, and $f$ is bounded as $\Im(\tau)\to\infty$. We denote with $M_k$ the space of modular forms of weight $k$.
\\\\
Equivalently, we can replace the functional equation in the previous definition with the equations
$$
    f(\tau)=f(\tau+1), \quad f(-1/\tau)=\tau^{k}f(\tau)
$$
as they generate the action of $SL_2(\Z)$ on the upper half plane. We denote $f(i\infty)=\lim_{\Im(\tau)\to\infty}f(\tau)$.
\\\\
As $f(\tau) = f(\tau + 1)$, it has an expansion in terms of $q = e^{2 \pi i \tau}$,
$$f(\tau) = \sum_{n = 0}^{\infty} a(n) q^n.$$
Writing $k = 12\ell + k'$, $k' \in \{0,4,6,8,10,14\}$, the \textit{Miller basis} of $M_k$ consists of the unique elements $f_{k,m}$ of $M_k$ with the $q$-expansion
$$f_{k,m} = q^m + O(q^{\ell+1})$$
for $m = 0,...,\ell$.
\\\\
An important example of a sequence of modular forms is the Eisenstein series. For $k \ge 4$ even we denote
$$E_{k}(\tau) = \frac{1}{2}\sum_{\gcd(c,d)=1}\frac{1}{(c\tau+d)^{k}}= 1+\gamma_{k}\sum_{n=1}^{\infty}\sigma_{k-1}(n)q^{n} \in M_k$$
where $\sigma_{s}(n)=\sum_{d\mid n}d^{s}$ and $\gamma_{k}=\frac{(2\pi i)^k}{\zeta(k) (k-1)!}=\frac{-2k}{\mathcal{B}_k} \in \Q$, where $\mathcal{B}_k$ is the $k$-th Bernoulli number.
\\\\
Another important example is the modular discriminant $\Delta \in M_{12}$, 
$$\Delta(\tau) = q \prod_{n = 1}^{\infty} (1-q^n)^{24} = q - 24q^2 + 252q^3 + ...,$$
which has a simple pole at $i\infty$ and does not vanish at any other point.
\newline
We also define the space of modular functions as holomorphic functions on the upper half plane that are invariant under the action of $SL_2(\Z)$ as defined for modular forms, where we allow a pole at the cusp $\Im(\tau)\to \infty$. By the same argument they can also be written as a sum in $q$, but can have negative powers of $q$ in the sum.
\\\\
Every modular function is a polynomial in the $j$-function
$$j = \frac{E_4^3}{\Delta} =\frac{1}{q} + 744 + 196884 q + ... $$
which is meromorphic at infinity. 
\\\\
Any modular form of weight  $k = 12\ell + k'$, $k' \in \{0,4,6,8,10,14\}$ can be written as
$$f = \Delta^{\ell} E_{k'}\cdot F_f(j)$$
where $F_f(t) = e_0 t^n + ... + e_n$ is called the Faber polynomial of $f$. One can show $\deg(F_f) = \ell - \mathrm{ord}_{\infty}(f)$, and we will denote the Faber polynomial of $f_{k,m}$ as $F_{k,m}$.
\\\\
This gives us a way to study the zeros of a modular form by studying the zeros of a polynomial.

\subsection{Computing the Faber Polynomial}\label{computing}
Take $f \in M_k(SL_2(\Z))$ and $m:=\mathrm{ord}_{\infty}(f)$.
\\\\
To compute the Faber polynomial of $f$, a modular form of weight $k$, we compare the Laurent series of $\frac{f}{\Delta^{\ell}E_{k'}}$ with a polynomial of $j$ of degree $\ell - m$. We can also compare the coefficients of the $q$-expansion of $f$ and $\Delta^\ell E_{k'} F(j)$ with $F$ of degree $\ell - m$.
\\\\
For example, we will look at the Faber polynomial of $f_{12\ell,\ell - 1} = q^{\ell - 1} + O(q^{\ell +1})$, which satisfies
$$q^{\ell - 1} + O(q^{\ell + 1}) = \Delta^{\ell} F_{12\ell, \ell-1}(j) = (q^{\ell} - 24 \ell q^{\ell + 1} + O(q^{\ell + 2}))(e_0 j + e_1).$$
As $j = \frac{1}{q} + 744 + O(q)$ we can write $e_0 j + e_1 = \frac{e_0}{q} + (e_0 744 + e_1) + O(q)$ and
$$\biggl( q^{\ell} - 24 \ell q^{\ell + 1} + O(q^{\ell + 2}) \biggr) \left(\frac{e_0}{q} + (e_0 744 + e_1) + O(q)\right) = e_0 q^{\ell - 1}  + (e_0 744 + e_1 - e_0 24 \ell) q^{\ell} + O(q^{\ell + 1})$$
and we must have $e_0 
= 1$ and $(e_0 744 + e_1 - e_0 24 \ell) = 0$, so
$$F_{12\ell, \ell-1}(t) = t + 24\ell - 744.$$
More generally, we can use the following algorithm - 
\\\\
Take $f\in M_k$ with $m = \mathrm{ord}_{\infty}(f)$
$$f \Delta^{-\ell}E_{k'}^{-1} = F_{f}(j) = \sum_{n = 0}^{\ell - m} e_n j^{\ell - m - n}.$$
For both sides, the non-zero exponent of $q$ with the smallest power is $q^{m - \ell}$. On the right-hand side, it is the coefficient of $q^{m - \ell}$ that is only determined by $e_0$.
\\\\
We choose $e_0$ to be the coefficient of $q^{m - \ell}$ in the left-hand side. 
\\\\
Assume we chose the $s$ leading coefficients. Then we look at
$$f \Delta^{-\ell}E_{k'}^{-1} - \sum_{n = 0}^{s-1} e_n j^{\ell - m - n} = \sum_{n = s}^{\ell - m} e_n j^{\ell - m - n}.$$
From our choices of $e_0,...,e_{s-1}$, the left-hand side is some constant times $q^{m + s - \ell} + O(q^{m + s -\ell + 1})$. The right-hand side is $e_s q^{m + s-\ell} + O(q^{m+s-\ell +1})$ based on the $q$-expansion of $j$. So we can choose $e_s$ to be the coefficient of $q^{m + s-\ell}$ in the left-hand side.
\\\\
Note that each coefficient of the Faber polynomial, $e_s$, can be expressed as a polynomial itself in the variables $k,m$ whose coefficients depend on $f$, $k'$.

\subsection{Additional Notation}
Faber polynomials will use the coefficient notation
$$F_f(t) = e_0 t^n + ... + e_n,$$
where for us the $i$-th coefficient is $e_i$. 

\section{Proof of the Linearity Theorem}
\subsection{Proof of Theorem 1}
Let $f_{k,m} = q^m + O(q^{\ell +1})$ be the $m$-th modular form in the Miller basis of weight $k = 12\ell + k'$ and let $F_{k,m}$ be its Faber polynomial. 
\newline
\newline
We will denote $\{x_i(f_{k,m})\}_{i = 0}^{\ell - m}$ the zeros of the Faber polynomial $F_{k,m}$, counted with multiplicity, and $S_n(f) := \sum_{i = 0}^{\ell - m}x_i^n$, where $x_i$ are the zeros of $f$. 
\begin{lemma}
    For $1 \le n \le \ell - m$, $S_n(f_{k,m}) = P_{n,k'}(k,m)$, a polynomial that depends on $n,k'$, but is a polynomial in the variables $k,m$, so if we fix $n,k'$, $S_n(k,m)$ is a polynomial in $k,m$.
\end{lemma}
\begin{proof}
    The coefficient of $t^{\ell - m -n}$ in the Faber polynomial can be written as a polynomial with coefficients that depends on $n,k'$ in the variables $k,m$. 
    \\\\
    Using Newton's identities, we can write the sum of the $n$-th powers of the zeros of the Faber polynomial $\sum_i x_i^n(k,m)$ as a polynomial with integer coefficients in the coefficients of the Faber polynomial.
\end{proof}
\begin{lemma}\label{linear poly}
    Take $k = 12\ell$. For all $1 \le n \le \ell - m$,
    $$P_{n,0}(k,m) = A_n \cdot k + B_n \cdot m.$$
\end{lemma}
\begin{proof}
    Consider $f_{12\ell, m}$ and note that the first $\ell - m + 1$ coefficients of the $q$-expansion of 
    $$f_{12\ell, m}^t = q^{tm} + O(q^{\ell + (t-1)m + 1})$$
    and 
    $$f_{12t\ell, tm} = q^{tm} + O(q^{t\ell + 1})$$
    are the same, for any integer $t>0$.
    \\\\
    This means that the first $\ell - m + 1$ coefficients of their Faber polynomials are the same, and by Newton's identities, they have the same sum of the $n$-th powers of their roots, for $1 \le n \le \ell - m$. Fix such $n$. As we are looking at $f_{12\ell, m}^t = \Delta^{t\ell} \cdot F_{12\ell, m}^t$, we have $t\cdot S_n(f_{12\ell,m}) = S_n((f_{12\ell,m}^t))$ and so
    $$t \cdot P_{n,0}(12\ell, m) = P_{n,0}(t\cdot 12\ell,tm),$$
    as $ F_{12\ell, m}^t$ has the same zeros as $ F_{12\ell, m}$ but with multiplicity multiplied by $t$.
    \newline
    \newline
    This implies the polynomial is linear and homogeneous, and we can write 
    $$S_n(f_{k,m}) =  A_n \cdot k + B_n \cdot m$$ 
    for some integer constants $A_n, B_n$.
\end{proof}
\begin{lemma}
    For all $1 \le n \le \ell - m$,
    $$\sum_{i=1}^{\ell - m} x_i^n = A_n \cdot k + B_n \cdot m + C_n(k')$$
    where $C_n(0) = C_n(4) = C_n(8) = 0$ and $C_n(6) = C_n(10) = C_n(14) = -\frac{1728^n}{2}$.
\end{lemma}
\begin{proof}
    Take $f_{12\ell + k', m}$ for some $k' \in \{4,6,8,10,14\}$. Note that the $\ell - m + 1$ leading coefficients of the $q$-expansion of 
    $$f_{k, m}^{t} = q^{tm} + O(q^{\ell + (t-1)m + 1})$$
    and 
    $$f_{t k, tm}$$
    are the same.
    \\\\
    As in the proof of Lemma \ref{linear poly}, this means that the $\ell - m + 1$ leading coefficients of the Faber polynomials of these two forms are the same, and they have the same sum of the $n$-th powers of their roots, for $1\le n \le \ell - m$. This implies that for $1 \le n\le \ell - m$,
    $$S_n(f_{k,m}^{t}) = S_n(f_{t k, t m}).$$
    We will fix such $n$, and look at the Faber polynomials of $f_{k, m}^{t}$ for the different values of $k'$.
    \subsection*{The case of $k' = 4,8$ - }
    We have $k = 12\ell + k'$, $k' = 4 w$, with $w = 1,2$ and take $t = 3$. Then 
    $$f_{k,m}^{3} = \Delta^{3 \ell} E_{4 w}^3 F_{k,m}^3(j) = \Delta^{3 \ell} \Delta^w j^w F_{k,m}^3(j) = \Delta^{3 \ell + w} j^w F_{k,m}^3(j),$$
    and so the Faber polynomial of $f_{k,m}^{3}$ is
    $$F_{f_{k,m}^{3}}(x) = x^w F_{k,m}^3(x).$$
    As multiplying by $x$ adds $0^n$ to the sum of the roots, we can write
    $$3 S_n(f_{k,m}) = S_n(f_{k,m}^{3}) = S_n(f_{3 k, 3 m}),$$
    and as $f_{3 k, 3 m}$ is a modular forms of weight divisible by 12, we use lemma \ref{linear poly}, and write 
    $$3 S_n(f_{k,m}) = S_n(f_{k,m}^{3}) = S_n(f_{3 k, 3 m}) = A_n \cdot  3 k + B_n \cdot 3 m,$$
    and
    $$ S_n(f_{k,m}) = A_n \cdot  k + B_n \cdot m,$$
    for $k' = 4,8$.
    \subsection*{The case of $k' = 6,10,14$ - }
    We have $k = 12\ell + k'$, $k' = 6 + 4w$ with $w = 0,1,2$. Take $t=6$. From the same argument as before, as $1728 \Delta = E_4^3 - E_6^2 =j\Delta - E_6^2$, 
    $$f_{k,m}^{6} = \Delta^{6 \ell} E_{4 w}^6 E_6^{6} F_{k,m}^6(j) = \Delta^{6 \ell} \Delta^{2w} j^{2w} \Delta^3 (j-1728)^3 F_{k,m}^6(j) = \Delta^{6 \ell + 2w + 3}  j^{2w} (j-1728)^3 F_{k,m}^6(j).$$
    This shows that the Faber polynomial of $f_{k,m}^{6}$ is
    $$F_{f_{k,m}^{6}} = x^{2w} (x-1728)^3 F_{k,m}^6(j),$$
    and comparing the sums of powers of zeros gives us 
    $$6 S_n(f_{k,m})  + 3 (1728)^n = S_n(f_{k,m}^{6}) = S_n(f_{6 k, 6 m}).$$
    We can use lemma \ref{linear poly} and write
    $$6 S_n(f_{k,m})  + 3 (1728)^n = S_n(f_{6 k, 6 m}) =  A_n \cdot  6 k + B_n \cdot 6 m,$$
    which gives us
    $$S_n(f_{k,m}) = A_n \cdot  k + B_n \cdot m - \frac{1728^n}{2},$$
    for $k' = 6,10,14$.
\end{proof}

So to finish proving Theorem \ref{linearity thm}, we only need to prove the formulas for the constants $A_n, B_n$.
\begin{remark}
    Define the normalised counting function of the zeros of $F_{k,m}$
    $$
    \mu_{k,m} = \frac{1}{\ell-m}\sum_{i=1}^{\ell - m}\delta_{x_i(k,m)}.
    $$
    This is a probability measure with $n$-th moment
    $$M_n(k,m) = 
    \frac{1}{\ell - m}\sum_{i=1}^{\ell - m} x_i^n,$$
    which we just showed is equal to
    $$M_n(k,m) = \frac{1}{\ell - m} (A_n\cdot k + B_n \cdot m + C_n(k')).$$
    This means that results about the distribution can be written in the language of $A_n, B_n, C_n(k')$.

\end{remark}
\begin{remark}
    Theorem \ref{linearity thm} gives a strong structure to the zeros of the Faber polynomials. We can see that in the limit as $k \to\infty$, the moments of the zeros only depend on the ratio $\frac{m}{\ell}$, as 
    $$M_n(k,m) = \frac{1}{\ell- m} (A_n \cdot k + B_n \cdot m + C_n(k')) =\frac{1}{1- \frac{m}{\ell}} (12\cdot A_n + B_n \cdot \frac{m}{\ell} + \frac{C_n(k') + k' A_n}{\ell}).$$
    As for $k' = 0$ we know $C_n(k') + k' A_n = 0$ and
    $$M_n(12\ell + k',m) = \frac{1}{1- \frac{m}{\ell}} (12 \cdot A_n + B_n \cdot \frac{m}{\ell}) $$
    and the average of the powers of the roots is fixed if we multiply both $\ell$ and $m$ by the same constant.
\end{remark}

\begin{corollary}\label{new formula for coefficients}
    We write $F_{k,m}(t) = e_0 t^{\ell - m} + e_1 t^{\ell - m - 1} + ... + e_{\ell - m}$ for the Faber polynomial.
    \\\\
    Then, using Newton's identities 
    $$e_n = \sum_{t_1 + 2 t_2 + ... + n t_n = n} \prod_{s = 1}^n (-1)^{t_s} \frac{(A_s \cdot k + B_s \cdot m + C_s(k'))^{t_s}}{t_s ! s^{t_s}}.$$
\end{corollary}
This gives us a non-recursive formula for the coefficients of the Faber polynomial.

\subsection{Computing the Values of $A_n$}
\begin{prop}\label{An}
    For all $1 \le n$, 
    $$A_n = \frac{1}{2\pi} \int_{\mathcal{A}} j^n(e^{i\theta}) d\theta.$$
    \end{prop}
\begin{proof}
    We will look at $f_{k,0}$, the first form in the Miller basis of weight $k = 12\ell$.
    \\\\
    We can write
    $$f_{k,0} = \Delta^{\ell} F_{k,0}(j).$$
    $\Delta$ has a zero only at $i\infty$ and it is a simple zero, and $j$ has a simple pole at $i\infty$. The zeros of $\Delta^\ell$ and the poles of $F_{k,0}(j)$ cancel out as the Faber polynomial is of degree $\ell$, and the form does not vanish at $i\infty$ (indeed, $f_{k,0} (\tau) = 1 + O(q^{\ell + 1})$ ).
    \\\\
    So the zeros of the form are $j^{-1}(x_i)$ where $x_i:=x_i(k,0)$ are the zeros of the Faber polynomial of $f_{k,0}$ counted with multiplicity.
    \\\\
    From Duke and Jenkins $\cite{first miller basis}$ we know that the zeros of the modular form $y_i := j^{-1}(x_i)$ are uniformly distributed on the arc of the fundamental domain as $k \to \infty$.
    $$A_n = \frac{1}{12}\lim_{k \to \infty}\frac{A_n \cdot k + B_n \cdot m + C_n(k')}{\ell} = \frac{1}{12}\lim_{k \to \infty} \frac{1}{\ell} \sum_{i = 1}^{\ell} x_i^n = \frac{1}{12}\lim_{k \to \infty} \frac{1}{\ell} \sum_{i = 1}^{\ell} (j(y_i))^n$$
    This is a Riemann sum and from the uniform distribution of $y_i$ in the limit we get
    $$\frac{1}{12}\frac{6}{\pi} \int_{\mathcal{A}} j^n(e^{i\theta}) d\theta = \frac{1}{2\pi} \int_{\mathcal{A}} j^n(e^{i\theta}) d\theta.$$
\end{proof}

\subsection{Computing the Values of $B_n$}
There are two ways to obtain the value of $B_n$. The first is an algebraic proof, which we will present here, and the second will be based on the limit distribution shown in Part 6. \newline \newline
We will use the notation $[x^n]g(x)$ for the coefficient of $x^n$ in $g(x)$, and $c(d) := [q^d]j(q)$. 
\begin{lemma}\label{coefficients of powers of j mod ell}
    For all $0 < i < n$,
    $$[q^0]j^n(q) = \sum_{d_1+...+nd_n = n} \frac{n!}{(n-\sum d_t)!} \prod_{t=1}^n \frac{(c(t-1))^{d_t}}{d_t !}$$
    $$[\ell^0][q^{n + m -\ell}]j^{\ell - m -i}(q) = \sum_{d_1+...+(n-i)d_{n-i} = n-i} \left((-1)^{\sum d_t} \prod_{\alpha = 0}^{\sum d_t - 1}(i + m + \alpha)\right) \left( \prod_{t=1}^{n-i} \frac{(c(t-1))^{d_t}}{d_t !}\right)$$
\end{lemma}
\begin{proof}
    We write the $q$-expansion
    $$qj(q) = 1 + 744q + 196884 q^2 + ... = \sum_{t = 0}^{\infty} c(t-1)q^t.$$
    Then
    $$[q^{n}](qj(q))^{N} = \sum_{d_1+...+nd_n = n} \frac{N!}{(N-\sum_{1\le i}d_i)!} \prod_{t=1}^n \frac{(c(t-1))^{d_t}}{d_t !}.$$
    By definition, $[q^{0}]j^{n}(q)= [q^n] (qj(q))^{n}$ and this is
    $$[q^{0}]j^{n}(q)= [q^n] (qj(q))^{n} = \sum_{d_1+...+nd_n = n} \frac{n!}{(n-\sum_{1\le i}d_i)!} \prod_{t=1}^n \frac{(c(t-1))^{d_t}}{d_t !}.$$
    In the same way, $[q^{n+m -\ell}]j(q)^{\ell - m - i} = [q^{n+m -\ell + \ell - m - i}](qj(q))^{\ell - m - i} = [q^{n-i}](qj(q))^{\ell - m -i}$ and this is
    $$[q^{n-i}](qj(q))^{\ell - m -i} = \sum_{d_1+...+(n-i)d_{n-i} = n-i} \frac{(\ell - m -i)!}{(\ell - m -i -\sum d_t)!} \prod_{t=1}^{n-i} \frac{(c(t-1))^{d_t}}{d_t !}.$$
    To find the coefficient of $\ell^0$,
    $$[\ell^0] \frac{(\ell - m -i)!}{(\ell - m -i -\sum d_t)!} = [\ell^0] (\ell - m -i) \cdot (\ell-m-i-1) \cdot ... \cdot (\ell - m- i -\sum d_t + 1) = (-1)^{\sum d_t} \prod_{\alpha = 0}^{\sum d_t - 1}(i + m + \alpha).$$
    Putting it all together, we get
    $$[\ell^0] [q^{n-i}](qj(q))^{\ell - m -i}  = \sum_{d_1+...+(n-i)d_{n-i} = n-i} \left((-1)^{\sum d_t} \prod_{\alpha = 0}^{\sum d_t - 1}(i + m + \alpha)\right) \left( \prod_{t=1}^{n-i} \frac{(c(t-1))^{d_t}}{d_t !}\right).$$
\end{proof}

\begin{prop}\label{Bn}
    For all $1 \le n$ we have
    $$B_n = -c_n(0)$$    
    where $c_n(0)$ is the coefficient of $q^0$ in $j^n$.
\end{prop}
\begin{proof}
    We will prove it using induction, so we will assume that $B_i = -c_i(0)$ for all $ i < n$.
    \newline \newline
    The proof is done in the following way - 
    \newline
    \newline
    We fix $1 \le n$, and we will compute $e_n = e_n(\ell,m)$, the coefficient of the Faber polynomial of $F_{12\ell, m}$, in two ways for a general choice of $\ell, m$ with $n \le \ell - m$. This will give us an equality $G_0(\ell,m) = G_1(\ell,m)$  when comparing the two computations, for two polynomials $G_0, G_1$ in two variables, $\ell, m$. So for any choice $0 \le m \le \ell - n$, the polynomial $G(\ell, m) := G_0(\ell,m) - G_1(\ell,m)$ vanishes. 
    \newline\newline
    Fix $0 \le m$. Then for all $\ell$ large enough, $G(\ell, m)$ vanishes, as a polynomial in one variable, $G(\cdot, m)$, and so the coefficient of $\ell^0$ (which is a polynomial in $m$) must be identically zero for all $m$. As it is zero for all $m$, it must be identically zero, and the coefficients of $\ell^0 m^1$ in both polynomials are the same.
    \newline \newline
    We will only compare the coefficients of  $\ell^0 m^1$ in $e_n$ for $0 < m$ and $1 \le n \le \ell - m$. 
    \newline
    \newline
    The first way we will compute the coefficient of the Faber polynomial is using Newton's identities on the sum of powers of the roots, and get for $1 \le n$,
    $$e_n \equiv \sum_{d_1 + ... + nd_n = n} \prod_{s = 1}^{n} (-1)^{d_s} \frac{(A_s \cdot 12 \ell + B_s m)^{d_s}}{d_s ! s^{d_s}} .$$
    Note that the coefficient of $\ell^0$ is
    $$\sum_{d_1 + ... + nd_n = n} \prod_{s = 1}^{n} (-1)^{d_s} \frac{(B_s m)^{d_s}}{d_s ! s^{d_s}} .$$
    The coefficient of $\ell^0 m^0$ is zero, and we get the coefficient of $\ell^0 m^1$ from the choice of $d_1 = ... = d_{n-1} = 0$ and $d_n = 1$, which leaves us with
    $$-m\frac{B_n}{n}.$$
    To compute $e_n$ in another way we will use the definition of the Faber polynomial. From section \ref{computing} we know that
    $$e_n = [q^{n+m -\ell}]\left( q^m \Delta^{-\ell} - \sum_{i = 0}^{n-1}e_i j^{\ell - m - i}\right).$$
    Starting with $\Delta^{-1}(q) = \sum_{i = -1}^{\infty}\tau_{-1}(i)q^i$, we can write
    $$[\ell^0]q^m \Delta^{-\ell} = [\ell^0]q^m \sum_{i = -1}^{\infty}(\tau_{-1}(i)q^i)^{\ell},$$
    and as we are looking for the coefficient of $q^{n+m-\ell}$,
    $$[q^{n+m -\ell}][\ell^0]\left( q^m \Delta^{-\ell}\right) = 0,$$
    for $0<n$. So we know 
    $$[\ell^0]e_n = [\ell^0][q^{n+m -\ell}]\left(- \sum_{i = 0}^{n-1}e_i j^{\ell - m - i}\right).$$
    From the fact that $m \vert [\ell^0]e_n$ for $1 \le n$,  
    $$[\ell^0 m^1 ]e_n = [q^{n+m -\ell}]\left(- \sum_{i = 1}^{n-1} (-\frac{B_i}{i}) [\ell^0 m^0] j^{\ell - m - i} - [\ell^0 m^1]j^{\ell - m}\right)$$
    $$ =  [q^{n+m -\ell}][\ell^0 m^0]\left(\sum_{i = 1}^{n-1} (\frac{B_i}{i})  j^{\ell - m - i}\right) - [\ell^0 m^1][q^{n+m -\ell}]j^{\ell - m}.$$
    Putting both together,
    $$- m\frac{B_n}{n} = m \cdot [\ell^0 m^0]\left(\sum_{i = 1}^{n-1} (\frac{B_i}{i})  [q^{n+m -\ell}]j^{\ell - m - i}\right) - m[\ell^0 m^1][q^{n+m -\ell}]j^{\ell - m}.$$
    Under the induction assumption, we will write
    $$ \frac{B_n}{n} = [\ell^0 m^0]\left(\sum_{i = 1}^{n-1} \frac{1}{i} [q^0]j^i(q)  [q^{n+m -\ell}]j^{\ell - m - i}\right) + [\ell^0 m^1][q^{n+m -\ell}]j^{\ell - m}.$$
    To simplify these we use lemma \ref{coefficients of powers of j mod ell}, which showed
    $$[\ell^0][q^{n + m -\ell}]j^{\ell - m -i}(q) = \sum_{d_1+...+(n-i)d_{n-i} = n-i} \left((-1)^{\sum d_t} \prod_{\alpha = 0}^{\sum d_t - 1}(i + m + \alpha)\right) \left( \prod_{t=1}^{n-i} \frac{(c(t-1))^{d_t}}{d_t !}\right).$$
    For $ i =0$, we want the coefficient of $\ell^0 m^1$, which is
    $$[\ell^0 m^1][q^{n + m -\ell}]j^{\ell - m}(q) = [m^1] \sum_{d_1+...+nd_{n} = n} \left((-1)^{\sum_t d_t} \prod_{\alpha = 0}^{\sum_t d_t - 1}(m + \alpha)\right) \left( \prod_{t=1}^{n-i} \frac{(c(t-1))^{d_t}}{d_t !}\right)$$
    $$=  \sum_{d_1+...+nd_{n} = n} \left((-1)^{\sum_t d_t} \prod_{\alpha = 1}^{\sum_t d_t - 1}\alpha\right) \left( \prod_{t=1}^{n-i} \frac{(c(t-1))^{d_t}}{d_t !}\right)$$
    $$= \sum_{d_1+...+nd_{n} = n} \left((-1)^{\sum_t d_t} (\sum_t d_t - 1)!\right) \left( \prod_{t=1}^{n-i} \frac{(c(t-1))^{d_t}}{d_t !}\right).$$
    For $ 0 < i$, we want the coefficient of $m^0$, which is
    $$[\ell^0 m^0][q^{n + m -\ell}]j^{\ell - m -i}(q) = [m^0]\sum_{d_1+...+(n-i)d_{n-i} = n-i} \left((-1)^{\sum d_t} \prod_{\alpha = 0}^{\sum d_t - 1}(i + m + \alpha)\right) \left( \prod_{t=1}^{n-i} \frac{(c(t-1))^{d_t}}{d_t !}\right)$$
    $$\sum_{d_1+...+(n-i)d_{n-i} = n-i} \left((-1)^{\sum d_t} \prod_{\alpha = 0}^{\sum d_t - 1}(i + \alpha)\right) \left( \prod_{t=1}^{n-i} \frac{(c(t-1))^{d_t}}{d_t !}\right).$$
    We will now treat $j(q) = \sum_{i = -1} c(i)q^i$ like a formal power series, and simplify these expressions more. We will use 
    $$-\log(1-z) = \sum_{r = 1}^{\infty} \frac{z^r}{r}$$
    and
    $$(1-z)^{-i} = \sum_{r=0}^{\infty} {i + r -1 \choose r}z^{r}. $$
    For $i=0$, we will look at 
    $$-\log(1 + (qj(q)-1)) = \sum_{r=1}^{\infty} (-1)^r \frac{(qj(q)-1)^r}{r} = \sum_{r=1}^{\infty} \frac{(-1)^r}{r} (\sum_{w = 1}^{\infty} c(w-1)q^w)^r $$
    $$= \sum_{r=1}^{\infty} \frac{(-1)^r}{r} \sum_{\sum_t d_t = r, d_t \geq 0} \frac{r!}{\prod_t d_t!}\prod_{1 \le t}c(t-1)^{d_t}q^{\sum_t td_t}.$$
    By reordering the terms, we have
    $$[\ell^0 m^1][q^{n+m -\ell}]j^{\ell - m} =  [q^n] \sum_{r=1}^{\infty} \frac{(-1)^r}{r} \sum_{\sum_t d_t = r, d_t \geq 0} \frac{r!}{\prod_t d_t!}\prod_{1 \le t}c(t-1)^{d_t}q^{\sum_t td_t},$$
    which is
    $$[\ell^0 m^1][q^{n+m -\ell}]j^{\ell - m} = - [q^n] \log(1 + (qj(q)-1)) = - [q^n] \log(qj(q)).$$
    For $0 < i$, 
    $$\frac{1}{(1 + (qj(q)-1))^i} = \sum_{r=0}^{\infty} (-1)^r {i + r -1 \choose r}(\sum_{w = 1}^{\infty} c(w-1)q^w)^r = \sum_{r=0}^{\infty} (-1)^r{i + r -1 \choose r} \sum_{\sum_t d_t = r, d_t \geq 0} \frac{r!}{\prod_t d_t!}\prod_{1 \le t}c(t-1)^{d_t}q^{\sum_t td_t},$$
    and by reordering the terms, we have
    $$[\ell^0 m^0][q^{n + m -\ell}]j^{\ell - m -i}(q) = [q^{n-i}]\frac{1}{(1 + (qj(q)-1))^i}.$$
    This gives us 
    $$ \frac{B_n}{n} = \sum_{i = 1}^{n-1} \frac{1}{i} [q^0]j^i(q)  [q^{n-i}](1 + (qj(q)-1))^{-i} - [q^n] \log(qj(q)).$$
    For $i = n$, $\frac{1}{i} [q^0]j^i(q)  [q^{n-i}](1 + (qj(q)-1))^{-i} = \frac{1}{n}[q^0]j^n(q)  [q^{0}](1 + (qj(q)-1))^{-n} = \frac{c_n(0)}{n}$, and 
    $$ B_n + c_n(0) = n \sum_{i = 1}^{n} \frac{1}{i} [q^0]j^i(q)  [s^{n-i}](1 + (qj(q)-1))^{-i} - n [q^n] \log(qj(q)),$$
    and so to finish the proof we need to show 
    $$\sum_{i = 1}^{n} \frac{1}{i} [q^0]j^i(q)  [q^{n-i}](1 + (qj(q)-1))^{-i} - [q^n] \log(qj(q))$$
    vanishes, which we will do in the following lemma.
    \begin{lemma} 
        $$\sum_{i = 1}^{n} \frac{1}{i} [q^0](j(q))^i  [q^{n-i}](qj(q))^{-i} = [q^n] \log(qj(q))$$
    \end{lemma}
    \begin{proof}
    Let $\phi(q)=qj(q)$, and $g$ such that $g(\frac{1}{j}(q)) = \frac{1}{j}(g(q)) = q$, the compositional inverse of $\frac{1}{j}$. From the Lagrange-Bürmann formula
    $$[z^n]\log\phi(g(z))=\frac1n[w^{n-1}]
    \frac{\phi'(w)}{\phi(w)}\phi(w)^n=\frac1{n^2}[w^{n-1}](\phi(w)^n)'=\frac1n[w^n]\phi(w)^n=\frac1n[q^0]j(q)^n. $$
    Taking $z=\frac{1}{j(q)}$,
    $$ \log(qj(q))=\sum_{i\ge1}\frac{1}{i}[w^0]j(w)^i \left(\frac{q}{qj(q)}\right)^i.$$
    Taking the coefficient of $q^n$, 
    $$[q^n]\log(qj(q))=\sum_{i=1}^{n}\frac{1}{i}[w^0]j(w)^i [q^{n-i}](qj(q))^{-i}. $$
    \end{proof}
    
\end{proof}

This finishes the proof of Theorem \ref{linearity thm}.
\section{Upper Bound}
We will now prove an upper bound for the value of $m$ for which all of the zeros of $f_{k,m}$ are on the arc.
\begin{lemma}\label{Bn integral}
    For all $1 \le n$, 
    $$B_n = -\int_{\frac{\pi}{3}}^{\frac{2\pi}{3}}\sin(\theta) j^n(e^{i\theta}) d\theta.$$
\end{lemma}
\begin{proof}
    From Proposition \ref{Bn}, we can write
    $$-B_n = c_n(0) = \Res_{q = 0} (q^{-1} j^n(q)) = \frac{1}{2\pi i}\int_{\abs{t} = r}(q^{-1} j^n(q)) dq$$
    for some $0 <r < 1$, where $c_n(0)$ is the coefficient of $q^0$ in $j^n$.
    \\\\
    After a change of variables $q = e^{2 \pi i \tau}$ with $A>1$, we have
    $$-B_n = \int_{-\frac{1}{2} + A i}^{\frac{1}{2} + A i} j^n(\tau) d\tau.$$ 
    Let $\gamma$ denote the following contour where $\gamma_t : [t-1, t] \to \C$ for $t = 1,...,4$,
    \begin{figure}[ht]
    \begin{center}
    \includegraphics[height=60mm]{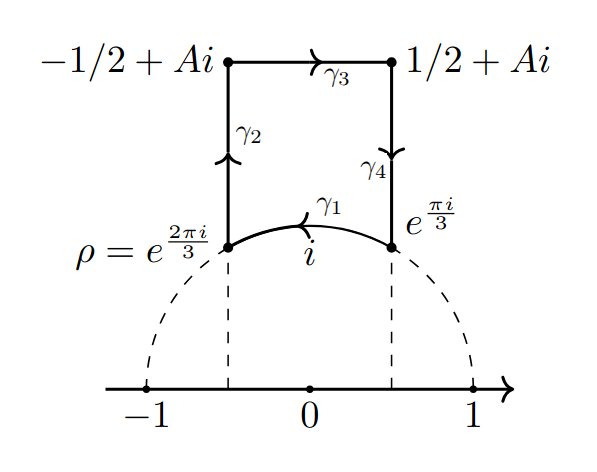}
    \caption{The contour of integration}
    \label{fig:diagram}
    \end{center}
    \end{figure}
    $$\gamma_1 (\tau) = e^{i\frac{\pi}{3}(1 + \tau)} \quad \quad \gamma_2 (\tau) = -\frac{1}{2} + \frac{\sqrt{3}}{2} i + (\tau-1) (A - \frac{\sqrt{3}}{2})i$$
    $$\gamma_3 (\tau) = -\frac{1}{2} + (\tau-2) + A i \quad \quad \text{and} \quad \quad \gamma_4 (\tau) = \frac{1}{2} + Ai + (\tau-3) (\frac{\sqrt{3}}{2} - A)i.$$

    As $j$ is a modular function, we have
    $$\int_{\gamma_2} j^n(\tau)d\tau = -\int_{\gamma_4} j^n(\tau)d\tau,$$
    which means that as $-B_n = \int_{\gamma_3} j^n(\tau)d\tau$,
    $$B_n = \int_{\gamma_1} j^n(\tau) d\tau.$$
    We will do a change of variables of $\tau = e^{i \theta}$ so $d\tau = i e^{i\theta} d\theta$, so 
    $$B_n = \int_{\frac{\pi}{3}}^{\frac{2\pi}{3}} ie^{i\theta} j^n(e^{i\theta}) d\theta = \int_{\frac{\pi}{3}}^{\frac{2\pi}{3}} i(\cos(\theta) + i\sin(\theta)) j^n(e^{i\theta}) d\theta.$$
    As for $0 \le \alpha \le \frac{\pi}{6}$, $\cos(\frac{\pi}{2} + \alpha) = -\cos(\frac{\pi}{2} - \alpha)$ and $j(e^{i(\frac{\pi}{2} + \alpha)}) = j(e^{i(\frac{\pi}{2} - \alpha)})$, the imaginary part of the integral vanishes and we are left with
    $$B_n = -\int_{\frac{\pi}{3}}^{\frac{2\pi}{3}}\sin(\theta) j^n(e^{i\theta}) d\theta.$$
\end{proof}
\begin{lemma}\label{pi approx} 
    $$\lim_{n \to \infty}\frac{A_n}{-B_n} = \frac{1}{4\pi}.$$
\end{lemma}
\begin{proof}
    From proposition \ref{An}, we know
    $$ A_n = \frac{1}{2\pi} \int_{\frac{\pi}{3}}^{\frac{\pi}{2}}j^{n}(e^{i\theta}) d\theta$$
    and from lemma \ref{Bn integral} and the symmetry of $j(e^{i\theta})$ and $\sin(\theta)$,
    $$B_n = -\int_{\frac{\pi}{3}}^{\frac{2\pi}{3}} \sin(\theta) j^n(e^{i\theta}) d\theta = -2\int_{\frac{\pi}{3}}^{\frac{\pi}{2}} \sin(\theta) j^n(e^{i\theta}) d\theta.$$
    We want to show that
    $$\lim_{n \to \infty}\frac{4 \pi A_n}{-B_n} = 1,$$
    where
    $$\frac{4 \pi A_n}{-B_n} = \frac{4 \pi \frac{1}{2\pi} \int_{\frac{\pi}{3}}^{\frac{\pi}{2}}j^{n}(e^{i\theta}) d\theta}{2\int_{\frac{\pi}{3}}^{\frac{\pi}{2}} \sin(\theta) j^n(e^{i\theta}) d\theta} = \frac{\int_{\frac{\pi}{3}}^{\frac{\pi}{2}}j^{n}(e^{i\theta}) d\theta}{\int_{\frac{\pi}{3}}^{\frac{\pi}{2}} \sin(\theta) j^n(e^{i\theta}) d\theta}.$$
    Note that as $\sin(\theta) \le 1$, we have $-B_n \le 4\pi A_n$, and so $\liminf_{n \to \infty}\frac{4 \pi A_n}{-B_n} \geq 1.$
    \\\\
    We will now show that the upper limit is bounded from above by 1.
    \\\\
    Denote $f(x) = \frac{j(e^{i(\frac{\pi}{6} x + \frac{\pi}{3})})}{1728}$, so $f :[0,1] \to [0,1]$, $f(0) = 0, f(1) = 1$ and as $j(e^{i\theta})$ is strictly increasing on $[\frac{\pi}{3}, \frac{\pi}{2}]$, so does $f$ on $[0,1]$. We will also denote $g(x) = \sin(\frac{\pi}{6} x + \frac{\pi}{3})$.
    \\\\
    Under that change of variable, we need to prove
    $$\limsup_{n \to \infty}\frac{\int_0^1 f^n(x) dx}{\int_0^1 g(x) f^n(x) dx} \le 1$$
    \\\\ 
    For all $\eps > 0$, we have $0 < \delta < 1$ for which $g(x) > 1-\eps$ for $1-\delta \le x \le 1$, and so
    $$\int_0^1 g(x) f^n(x) dx \geq \int_{1-\frac{\delta}{2}}^1 (1-\eps) f^n(x) dx.$$
    As $f$ is strictly increasing, we have
    $$\int_{1-\frac{\delta}{2}}^1 (1-\eps) f^n(x) dx \geq \frac{\delta}{2}(1-\eps) f^n(1-\frac{\delta}{2}).$$
    As for the numerator, 
    $$\int_0^1 f^n(x) dx = \int_0^{1-\delta} f^n(x) dx + \int_{1-\delta}^1 f^n(x) dx \le (1-\delta) f^n(1-\delta) + \int_{1-\delta}^1 f^n(x) dx$$
    Putting it all together, 
    $$\limsup_{n \to \infty}\frac{\int_0^1 f^n(x) dx}{\int_0^1 g(x) f^n(x) dx} = \limsup_{n \to \infty}\frac{\int_0^{1-\delta} f^n(x) dx + \int_{1-\delta}^1 f^n(x) dx}{\int_0^1 g(x) f^n(x) dx} \le \limsup_{n \to \infty}\frac{\int_0^{1-\delta} f^n(x) dx + \int_{1-\delta}^1 f^n(x) dx}{(1-\eps)\int_{1-\delta}^1 f^n(x) dx} $$
    $$= \frac{1}{1-\eps}(1 + \limsup_{n \to \infty}\frac{\int_{0}^{1-\delta} f^n(x) dx}{\int_{1-\delta}^1 f^n(x) dx})\le \frac{1}{1-\eps}(1+ \frac{(1-\delta)f^n(1-\delta)}{\frac{\delta}{2}f^n(1-\frac{\delta}{2})}).$$
    As $f$ is strictly increasing, 
    $$\limsup_{n \to \infty}\frac{(1-\delta)f^n(1-\delta)}{\frac{\delta}{2}f^n(1-\frac{\delta}{2})} = 0,$$ 
    and so for all $\eps > 0$,
    $$\limsup_{n \to \infty}\frac{4 \pi A_n}{-B_n} \le \frac{1}{1-\eps},$$
    and $1 \le \liminf_n \frac{4 \pi A_n}{-B_n} \le \limsup_n\frac{4 \pi A_n}{-B_n} \le 1$,  and the limit exist and is 1, as required.
\end{proof}

\begin{prop}\label{not on arc range}
    For $f_{k,m}$ with $1 \le n \le \ell$, if $\frac{A_n}{-B_n} k + \frac{C_n(k')}{-B_n} < m \le \ell - n$, at least one of the zeros of $f_{k,m}$ is not on the arc.
\end{prop}
\begin{proof}
    We know that for $1 \le m \le \ell - n$, the sum of the powers of the zeros is
    $$\sum_{i=1}^{\ell-m} x^n_i = A_n \cdot k + B_n \cdot m + C_n(k').$$
    If the sum of the zeros is negative, at least one of the zeros of the polynomial is not in $[0,1728]$, and the corresponding zero of the form is not on the arc.
    Note that the sum is negative if and only if
    $\frac{A_n}{-B_n} k + \frac{C_n(k')}{-B_n} < m$.
    If $n$ is even, we also know two of the zeros are complex.
    \\\\
    Overall, this means that if $\frac{A_n}{-B_n} k + \frac{C_n(k')}{-B_n} < m \le \ell - n$ at least one of the zeros is not on the arc.
\end{proof}
For example, we can take $n=1$ and get the following - 
\begin{corollary}\label{30/31}
    If $k' = 0,4,8$ and $\frac{30}{31}\ell + k'\frac{5}{62} < m \le \ell - 1$, at least one of the zeros of $f_{k,m}$ is not on the arc.
    \\\\
    If $k' = 6,10,14$ and $\frac{30}{31} \ell + k' \frac{5}{62} - \frac{36}{31} < m \le \ell - 1$, at least one of the zeros of $f_{k,m}$ is not on the arc.
\end{corollary}
Additionally, if we take $m = \ell - n$,
\begin{corollary}
    Fix $n > 0$. For all $\ell$ such that $\ell > \frac{A_n k' + C_n(k') - n B_n}{-B_n-12A_n}$, at least one of the zeros of $f_{k,\ell - n}$ is not on the arc. 
\end{corollary}
\begin{proof}
    Take $m = \ell - n$. To use Proposition \ref{not on arc range} we need $\frac{A_n}{-B_n} k + \frac{C_n(k')}{-B_n} < m = \ell - n \le \ell-n$.
    This is equivalent to
    $$\frac{A_n}{-B_n} (12\ell + k') + \frac{C_n(k')}{-B_n} < \ell - n$$
    and
    $$\frac{A_n}{-B_n} k' + \frac{C_n(k')}{-B_n} + n < \ell - 12\frac{A_n}{-B_n}\ell = \left(1-12\frac{A_n}{-B_n}\right)\ell.$$
    Which is equivalent to
    $$\ell > \frac{\frac{A_n}{-B_n} k' + \frac{C_n(k')}{-B_n} + n}{1-12\frac{A_n}{-B_n}} = \frac{A_n k' + C_n(k') - n B_n}{-B_n-12A_n}.$$
\end{proof}
\begin{corollary}\label{pi upper bound}
    Let $\frac{3}{\pi}<c<1$ and let $\ell > m = m(\ell) > c \cdot \ell$.
    Then for $\ell$ large enough, at least one of the zeros of $f_{k,m}$ is not on the arc.
\end{corollary}
\begin{proof}
    The case of $\ell - m $ being bounded as $\ell \to \infty$ was shown by Rudnick \cite{large m}, so we will assume $\ell - m \to \infty$ as $\ell \to \infty$.
    \\\\
    We will take $n \to \infty$ by taking $n = n(\ell, m) = \ell - m -1$, which means that the right inequality in proposition \ref{not on arc range} is satisfied
    $$\frac{A_n}{-B_n} k + \frac{C_n(k')}{-B_n} < m \le \ell - n.$$
    As $C_n(k'), \frac{C_n(k')}{-B_n}$ are non-positive and  $c \cdot \ell < m$, it is enough to show that for $\ell$ large enough,
    $$\frac{A_n}{-B_n} k < c \cdot\ell,$$
    which is equivalent to
    $$\frac{A_n}{-B_n} k' < \left(c-12\frac{A_n}{-B_n}\right)\ell.$$
    Take $\eps > 0$ with $\frac{3}{\pi} + \eps < c$. As $n \to \infty$, $12\frac{A_n}{-B_n} \to \frac{3}{\pi}$ so for $n$ large enough, $12\frac{A_n}{-B_n} < \frac{3}{\pi} + \eps < c$, and so  
    $$\left(c-12\frac{A_n}{-B_n}\right) > 0,$$
    and for large enough $\ell$ the inequality is satisfied.
\end{proof}

\section{Limit Distribution and Proof of Theorem 2}
As we are interested in the distribution of the zeros of the Faber polynomials, we think of the sums in the following way. We denote $\mu_{k,m} = \frac{1}{\ell-m} \sum_{i = 1}^{\ell-m} \delta_{x_i(k,m)}$, and so the moments of the measure are
$$M_{n}(k,m) = \int x^n d\mu_{k,m} = \frac{1}{\ell-m} \sum_{i = 1}^{\ell-m} x_i^{n}(k,m),$$
which is the average of the powers of the zeros. Fixing $1 \le n$, we will show that the limit of $M_n(k,m)$ as $k \to \infty$ will give us the $n$-th moment of the limit distribution of the zeros.
\\\\
We will give two proofs. The first proof will prove Theorem \ref{limit dist} for $c < \frac{2}{9}$, and the second proof will prove it for $0 < c < 1$ for which $m \sim c \cdot \ell$ with all the zeros of $f_{k,m}$ on the arc. 
\begin{proof}[First proof of Theorem \ref{limit dist}]
    Raveh \cite{Roei's paper} showed that for $h(\theta) =  \frac{k}{2} \theta + 2\pi m \cos(\theta)$, there is a unique zero of the modular form $f_{k,m}$ between every two values of $h$ which are integer multiples of $\pi$. For $c < \frac{2}{9}$, $h'$ is non-negative and is increasing in $[\frac{\pi}{2}, \frac{2 \pi}{3}]$. 
    \\\\
    The number of zeros of $f_{k,m}(e^{i\theta})$ for $\theta_1 \le \theta \le \theta_2$ is the number of integer values of $\frac{h}{\pi}$ in this interval with error bounded by 2, which is $\floor{(h(\theta_2) - h(\theta_1))/\pi} = \frac{h(\theta_2) - h(\theta_1) + O(1)}{\pi}.$
    \\\\
    For such $c$, all the zeros of the modular form are on the arc, so there are $\ell - m$ such zeros.
    $$\lim_{\ell \to \infty} \frac{h(\theta_2) - h(\theta_1) + O(1)}{\pi(\ell - m)} = \lim_{\ell \to  \infty} \frac{(6 \ell + \frac{k'}{2} )\theta_2 + 2\pi c \cdot \ell \cos(\theta_2) - (6 \ell + \frac{k'}{2} ) \theta_1 + 2\pi c \cdot \ell \cos(\theta_1) + O(1)}{\pi (1-c) \ell} = $$
    $$\frac{6}{\pi (1-c)} (\theta_2 - \theta_1) + \frac{2 c}{1-c} (\cos(\theta_2) - \cos(\theta_1)).$$
\end{proof}

\begin{proof}[Second proof of theorem \ref{limit dist}]
Fix a sequence $m=m(\ell)$ such that $m \sim c \cdot \ell$,$0<c<1$, and all the zeros of $f_{k,m}$ in the fundamental domain lie on the arc. Then all the roots of the Faber polynomial $F_{k,m}$ lie in $[0,1728]$, so the probability measures
$$
\mu_{k,m}=\frac{1}{\ell-m}\sum_{i=1}^{\ell-m}\delta_{x_i(k,m)}
$$
are supported on $[0,1728]$. 
\newline
Using Theorem \ref{linearity thm}, for each fixed $1 \le n$ and all sufficiently large $\ell$,
$$ M_n(k,m) = \int x^n\,d\mu_{k,m}(x) =\frac{1}{\ell - m}\left(A_n \cdot k+B_n \cdot m+C_n(k')\right). $$
We can take the limit, and get 
$$ m_n:=\lim_{\ell\to\infty}M_n(k,m) =\frac{12}{1-c} A_n+\frac{c}{1-c}B_n.$$
The limit distribution is a probability measure $\mu$ supported on $[0,1728]$ with these moments, as convergence of moments gives convergence of integrals of polynomials, and uniform polynomial approximation of continuous functions on $[0,1728]$ gives weak convergence of $\mu_{k,m}$ to $\mu$.
\newline 
\newline
For $\frac{\pi}{2}\le\theta\le\frac{2\pi}{3}$, we define 
$ J(\theta):=j(e^{i\theta})$, and $\Theta:[0,1728]\to[\frac{\pi}{2},\frac{2\pi}{3}]$ its inverse.
Denote
$$ f_c(x):= \ind_{[0 < x < 1728]} \frac{\frac6\pi-2c\sin\Theta(x)} {(1-c)(-J'(\Theta(x)))}.$$
From Corollary \ref{pi upper bound}, we know  $c\le\frac{3}{\pi}$ and so $f_c(x)$ is non-negative.
$$
\begin{aligned}
\int_0^{1728}f_c(x)\,dx
&=-\int_{\frac{\pi}{2}}^{\frac{2\pi}{3}}
f_c(J(\theta))J'(\theta)\,d\theta\\
&=\frac1{1-c}\int_{\frac{\pi}{2}}^{\frac{2\pi}{3}}
\left(\frac{6}{\pi}-2c\sin\theta\right)d\theta\\
&=\frac1{1-c}\left(1-2c\cdot\frac{1}{2}\right) =1.
\end{aligned}
$$
This proves $f_c(x)$ defines a probability measure on $[0,1728]$. Computing the moments, we get 
$$
\begin{aligned}
\int_0^{1728}x^n f_c(x)dx =-\int_{\frac{\pi}{2}}^{\frac{2\pi}{3}} J(\theta)^n f_c(J(\theta))J'(\theta)\,d\theta
=\frac{6}{\pi(1-c)} \int_{\frac{\pi}{2}}^{\frac{2\pi}{3}}j^n(e^{i\theta})\,d\theta -\frac{2c}{1-c} \int_{\frac{\pi}{2}}^{\frac{2\pi}{3}} \sin\theta j^n(e^{i\theta}) d\theta.
\end{aligned} $$
From the integral formulas for $A_n$ and $B_n$, this is
$$
\frac{12}{1-c} A_n+\frac{c}{1-c}B_n = m_n.$$
For $\frac{\pi}{2}\le\theta_1\le\theta_2\le\frac{2\pi}{3}$,the limiting probability is
$$ \begin{aligned}
\int_{J(\theta_2)}^{J(\theta_1)}f_c(x)\,dx
&=\int_{\theta_1}^{\theta_2}
f_c(J(\theta))(-J'(\theta))\,d\theta\\
&=\int_{\theta_1}^{\theta_2}
\frac{\frac6\pi-2c\sin\theta}{1-c}\,d\theta\\
&=\frac6{\pi(1-c)}(\theta_2-\theta_1)
+\frac{2c}{1-c}
\bigl(\cos\theta_2-\cos\theta_1\bigr).
\end{aligned} $$
\end{proof}

\begin{remark}\label{stretching the dist}
The second proof uses the moments to compute the distribution. The distribution computed from the moments can only be uniquely determined for all such $0 \le c$ for which all the zeros of the modular form are on the arc.
\\\\
As long as all the zeros of the Faber polynomial are in $[0,1728]$, we can use the fact that the Hausdorff moment problem has a unique solution, if one exists. We showed $f_c$ is a function that generates the correct moments. This means that $f_c$ is the density function of the zeros of the Faber polynomials long as $0 \le c$ is small enough such that all the zeros of the Faber polynomial of $f_{k,m}$ are in $[0,1728]$ (which implies the moments determine the distribution), and as long as $f_c$ non-negative (which is needed for it to be a density function). 
\\\\
For $f_c$ to be non-negative it is enough if $c \le \frac{3}{\pi}$, as for such c, the probability that a zero is in $[\theta_1, \theta_2]$ is
$$\frac{6}{\pi (1-c)} (\theta_2 - \theta_1) + \frac{2 c}{1-c} (\cos(\theta_2) - \cos(\theta_1)) = \frac{2}{1-c} (\frac{3}{\pi } (\theta_2 - \theta_1) + c (\cos(\theta_2) - \cos(\theta_1)))$$
$$\ge \frac{6}{(1-c) \pi } (\theta_2 - \theta_1 +\cos(\theta_2) - \cos(\theta_1)).$$
As $(x + \cos(x))' = 1-\sin(x) \ge 0$, the probability is indeed non-negative and $f_c$ is a density function. 
\\\\
This means that for any proof that extends the range of $c \le \frac{3}{\pi}$ for which all of the zeros are on the arc, the limit distribution is as above. 
\\\\
For $c >\frac{3}{\pi}$, at least one of the zeros is not on the arc and so the limit distribution might not be uniquely determined by the moments, and either more details or a new argument are required.
\end{remark}


\begin{thebibliography}{99}

\bibitem{Eisenstein zeros}
F. K. C. Rankin , H. P. F. Swinnerton-Dyer, \textit{On the Zeros of Eisenstein Series}, Bulletin of the London Mathematical Society, Volume 2, Issue 2, July 1970, Pages 169–170, https://doi.org/10.1112/blms/2.2.169


\bibitem{large m}
Zeév Rudnick, \textit{Zeros of modular forms and Faber polynomials}, Mathematika, vol 70 (2024), no. 2 \url{https://doi.org/10.1112/mtk.12244}

\bibitem{first miller basis}
William Duke and Paul Jenkins, \textit{On the zeros and coefficients of certain weakly holomorphic modular forms}, Pure and Applied Mathematics Quarterly, Q4 (2008), no. 4, part 1,1327–1340

\bibitem{delta}
M. Kaneko and Y. Sakai \textit{The Ramanujan-Serre differential operators and certain elliptic curves} Proceedings of the American Mathematical Society 141 (2013), no.10, 3421–3429.

\bibitem{Roei's paper}
Roei Raveh, \textit{On the Zeros of the Miller Basis of Cusp Forms} 2024. \url{https://arxiv.org/pdf/2405.01184}


\end{thebibliography}
\end{document}